\documentclass{article}
\usepackage{ijcai22}

\usepackage{times}
\usepackage{soul}
\usepackage[hyphens]{url}
\usepackage[hidelinks]{hyperref}
\usepackage[utf8]{inputenc}
\usepackage[small]{caption}
\usepackage{graphicx}
\usepackage{amsmath}
\usepackage{amsthm}
\usepackage{booktabs}
\usepackage{algorithm}
\usepackage{algorithmic}
\usepackage{adjustbox}
\usepackage[caption=false]{subfig}

\graphicspath{{figures/}}

\usepackage[lighttt]{lmodern}
\ttfamily
\DeclareFontShape{OT1}{lmtt}{m}{it}
{<->sub*lmtt/m/sl}{}
\usepackage{listings}
\usepackage{opl}
\usepackage{varwidth}

\usepackage[colorinlistoftodos]{todonotes}

\usepackage{xcolor,soul}
\definecolor{dr}{rgb}{0.75,0.00,0.00}
\definecolor{lr}{rgb}{1.00,0.75,0.75}
\sethlcolor{lr}

\usepackage{multirow}
\usepackage{multicol}

\newtheorem{proposition}{Proposition}

\title{Optimization Models for Autonomous Transfer Hub Networks}

\author{
Chungjae Lee\and
Kevin Dalmeijer\and
Pascal Van Hentenryck
\affiliations
H. Milton Stewart School of Industrial and Systems Engineering, Georgia Institute of Technology
\emails
clee384@gatech.edu,
dalmeijer@gatech.edu,
pascal.vanhentenryck@isye.gatech.edu
}

\begin{document}

\maketitle

\begin{abstract}
Autonomous trucks are expected to fundamentally transform the freight transportation industry. In particular, Autonomous Transfer Hub Networks (ATHN), which combine autonomous trucks on middle miles with human-driven on the first and last miles, are seen as the most likely deployment pathway of this technology. 
This paper presents three methods to optimize ATHN operations and compares them: a constraint-programming model, a column-generation approach, and a bespoke network flow method. Results on a real case study indicate that the network flow model is highly scalable and outperforms the other two approaches by significant margins. 
\end{abstract}

\section{Introduction}

Autonomous trucks are expected to fundamentally transform the freight transportation industry. Morgan Stanley estimates the potential savings from automation at \$168 billion annually for the US alone \cite{Greene2013-AutonomousFreightVehicles}.
Additionally, autonomous transportation may improve on-road safety, and reduce emissions and traffic congestion \cite{ShortMurray2016-IdentifyingAutonomousVehicle,SlowikSharpe2018-AutomationLongHaul}.

SAE International~\shortcite{SAEInternational2018-TaxonomyDefinitionsTerms} defines different levels of driving automation, ranging from L0 to L5, corresponding to no-driving automation to full-driving automation.
The current focus is on L4 technology (high automation), which aims at delivering automated trucks that can drive without any need for human intervention in specific domains, e.g., on highways.
The trucking industry is actively involved in making L4 vehicles a reality.
Daimler Trucks, one of the leading heavy-duty truck manufacturers in North America, is working with both Torc Robotics and Waymo to develop autonomous trucks \cite{HDT2021-DaimlersRedundantChassis}.
In 2020, truck and engine maker Navistar announced a strategic partnership with technology company TuSimple to go into production by 2024 \cite{TransportTopics2020-NavistarTusimplePartner}.
Truck manufacturers Volvo and Paccar have both announced partnerships with Aurora \cite{TechCrunch2021-AuroraVolvoPartner}.
Other companies developing self-driving vehicles include Embark, Gatik, Kodiak, and Plus \cite{FleetOwner2021-TusimpleAutonomousTruck,Forbes2021-PlusPartnersIveco,FreightWaves2021-GatikIsuzuPartner}.

A study by Viscelli \shortcite{Viscelli-Driverless?AutonomousTrucks} describes different scenarios for the adoption of autonomous trucks by the industry.
The most likely scenario, according to some of the major players, is the \emph{transfer hub business model} ~\cite{Viscelli-Driverless?AutonomousTrucks,RolandBerger2018-ShiftingGearAutomation,ShahandashtEtAl2019-AutonomousVehiclesFreight}.
An Autonomous Transfer Hub Network (ATHN) makes use of autonomous truck ports, or \emph{transfer hubs}, to hand off trailers between human-driven trucks and driverless autonomous trucks.
Autonomous trucks then carry out the transportation between the hubs, while conventional trucks serve the first and last miles (see Figure~\ref{fig:autonomous_example}). 
Orders are split into a first-mile leg, an autonomous leg, and a last-mile leg, each of which served by a different vehicle.
A human-driven truck picks up the cargo at the customer location, and drops it off at the nearest transfer hub.
A driverless autonomous truck moves the trailer to the transfer hub closest to the destination, and another human-driven truck performs the last leg.

\begin{figure}[!t]
	\centering
	\includegraphics[width=\linewidth]{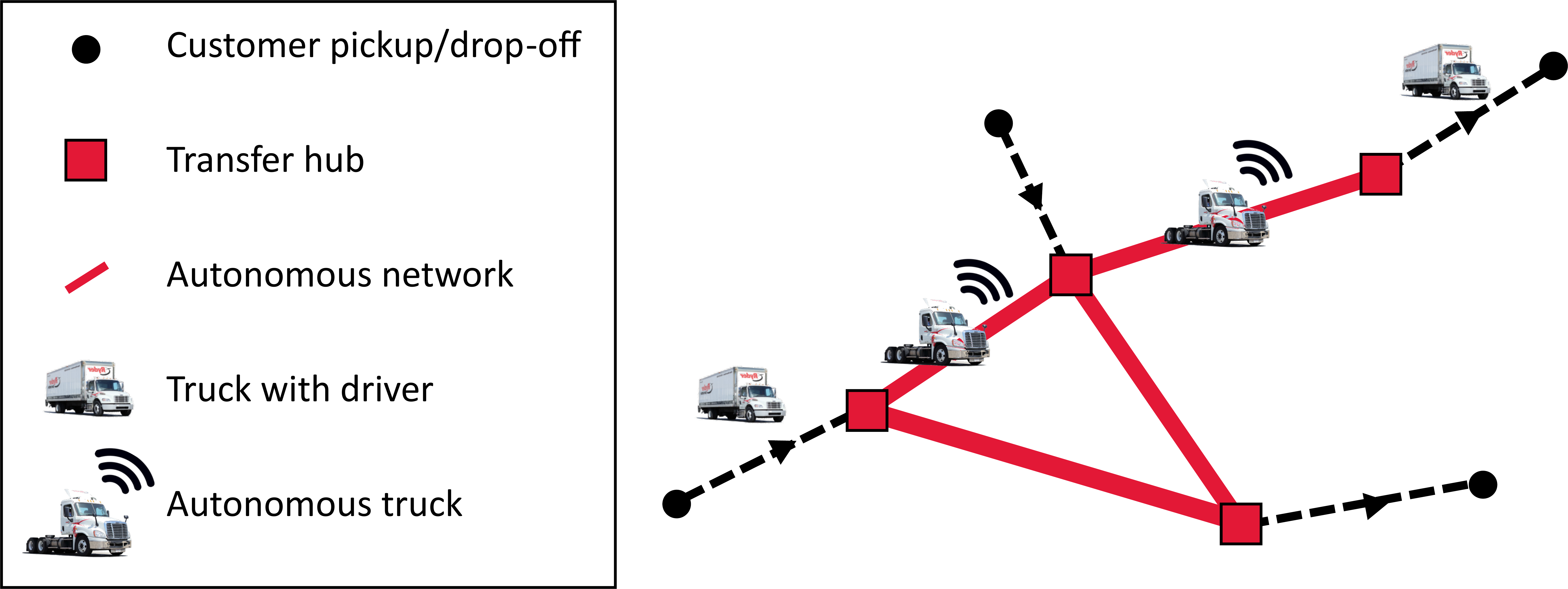}
	\caption{Example of an Autonomous Transfer Hub Network.}
	\label{fig:autonomous_example}
\end{figure}

ATHNs apply automation where it counts: Monotonous highway driving is automated, while more complex local driving and customer contact is left to humans.
Global consultancy firm Roland Berger \shortcite{RolandBerger2018-ShiftingGearAutomation} estimates operational cost savings between 22\% and 40\% in the transfer hub model, based on cost estimates for three example trips.
A recent white paper published by Ryder System, Inc. \shortcite{RyderSAM2021-ImpactAutonomousTrucking} studies whether these savings can be attained for actual operations and realistic orders.
It models ATHN operations as a scheduling problem and uses a Constraint Programming (CP) model \cite{DalmeijerVanHentenryck2021-OptimizingFreightOperations} to minimize empty miles and produce savings from 27\% to 40\% on a case study in the Southeast of the US.

This paper reconsiders the optimization of ATHNs through both a solution quality and a solver performance lens. In addition to the CP model, it considers a model based on the Vehicle Routing Problem with Full Truckloads (VRPFL) \cite{ArunapuramEtAl2003-VehicleRoutingScheduling}. This model is tackled by two different approaches: a Column Generation (CG)
approach and a Network Flow (NF) approach. The main technical contributions can be summarized as follows: (1) the NF model provides lower bounds as strong as those of CG on the Ryder case study; (2) solutions to the NF model can be transformed into upper bounds that are within 1\% of optimality; and (3) the resulting NF-based approach is highly scalable and provides orders of magnitudes of improvement over CP and CG. 
From a case study standpoint, the main contribution of the paper  is to provide the first lower bounds for the ATHN optimization and to demonstrate that real instances can be solved to near-optimality. 

The remainder of this paper is organized as follows.
Section~\ref{sec:problem} formally defines the problem of optimizing ATHN operations, and Section~\ref{sec:models} introduces the three solution methods.
Section~\ref{sec:casestudy} compares the methods on the case study, and Section~\ref{sec:conclusion} provides the conclusions.

\section{Problem Statement}
\label{sec:problem}
The ATHN problem aims to find a plan for a set of vehicles to fulfill a list of customer tasks at minimum cost.
Let $L$ be a set of locations, including hub locations $L_H \subset L$, with driving time $\tau_{ij} \ge 0$ and driving distance $c_{ij} \geq 0$ between locations $i, j \in L$ ($\tau_{ij}, c_{ij} > 0$ if $i \neq j$). Customers request freight to be picked up at a given location at a given time, to be transported to a dropoff location.
This information is used to generate a set of tasks $T$, where each order generates a first-mile task (regular truck from pickup to hub), an autonomous task (transportation between the hubs), and a last-mile task (regular truck from hub to dropoff).
Every task $t\in T$ corresponds to a single leg, and is defined by an origin $o(t) \in L$, a destination $d(t) \in L$, and a pickup time $p(t)$.
The duration of a task equals $\tau_{o(t)d(t)} + 2S$, where $S \ge 0$ is the fixed time for loading or unloading a trailer.
The pickup time of the first-mile task is equal to the customer requested pickup time, and pickup times for subsequent tasks are set to the time that the freight is planned to arrive at $o(t)$.
The set of available trucks is partitioned into autonomous trucks $K$, and regular trucks $K_h$ at every hub $h \in L_H$.

A feasible plan is created by assigning all tasks to the vehicles.
It is assumed that an appointment flexibility of $\Delta \geq 0$ minutes is permitted, which means that task $t\in T$ may start anywhere in the interval $[p(t)-\Delta, p(t)+\Delta]$.
Trucks are assumed to provide dedicated service, i.e., they transport one order at a time, and may be scheduled with any start- and end-point over the planning horizon.
Every task must be assigned to the right kind of truck, and tasks performed by the same vehicle must not overlap in time.
If the dropoff location $i$ of the previous task does not match the pickup location $j$ of the next task, then an empty relocation is necessary with time $\tau_{ij}$ and cost $c_{ij}$.
An optimal plan performs all the tasks at minimum total driving distance, or equivalently, at minimum total relocation distance.

The problem described above can be decomposed and solved independently for the autonomous network and for the operations at each of the hubs.
The main challenge is in optimizing the autonomous network since, in practice, the first- and last-mile problems are not very constrained. That is the focus of the paper.

\section{Models and Methodology}
\label{sec:models}
The problem of optimizing ATHN operations is modeled as a scheduling problem and as a VRPFL.
Three different solution methods are proposed: The scheduling problem is solved with CP, and the VRPFL is addressed with a CG-based heuristic and with an NF model.
For simplicity the methods are presented for the autonomous part of the network only.
This is without loss of generality, because the ATHN problem is decomposable.

\subsection{Scheduling Modeling}
\label{sec:CP}

The ATHN optimization can be modeled as a scheduling problem and solved using  CP as in the Ryder \shortcite{RyderSAM2021-ImpactAutonomousTrucking} case study.  
Figure~\ref{fig:formulation} reproduces the CP model for this problem using OPL syntax \cite{VanHentenryck1999-OplOptimizationProgramming}.

\newsavebox{\modelbox}
\begin{lrbox}{\modelbox}
\begin{varwidth}{0.46\textwidth}
\vspace{-0.25cm}
\begin{lstlisting}
range Trucks = ...;|\label{cp:range_start}|
range Tasks = ...;
range Sites = ...;
range Horizon = ...;
range Types = Sites union { shipType };|\label{cp:range_end}|
int or[Tasks] = ...;|\label{cp:param_start}|
int de[Tasks] = ...;
int pickupTime[Tasks] = ...;|\label{cp:pickup_time}|
int loadTime = ...;|\label{cp:load_time}|
int flexibility = ...;|\label{cp:flex}|
int travelTime[Types,Types] = ...;
int travelCost[Types,Types] = ...;|\label{cp:param_end}|

dvar interval task[t in Tasks] in Horizon
  size travelTime[or[t],de[t]] + 2*loadTime;
dvar interval ttask[k in Trucks,t in Tasks] optional in Horizon
  size travelTime[or[t],de[t]] + 2*loadTime;
dvar interval load[Trucks,Tasks] optional in Horizon size loadTime;
dvar interval ship[k in Trucks,t in Tasks] optional in Horizon
  size travelTime[ort],de[t]];
dvar interval unload[Trucks,Tasks] optional in Horizon size loadTime;
dvar sequence truckSeq[k in Trucks]
  in append(all(t in Tasks)load[k,t],all(t in Tasks)ship[k,t],all(t in Tasks)unload[k,t])
  types append(all(t in Tasks)or[t],all(t in Tasks)shipType,all(t in Tasks)de[t]);
dvar int emptyMilesCost[Trucks,Tasks];
dvar int truckEmptyMilesCost[Trucks];

minimize sum(k in Trucks) truckEmptyMilesCost[k];|\label{cp:obj}|

constraints {
  forall(t in Tasks)|\label{cp:flex_constr_start}|
    startOf(task[t]) >= pickupTime[t] - flexibility;
    startOf(task[t]) <= pickupTime[t] + flexibility;|\label{cp:flex_constr_end}|
  forall(k in Trucks,t in Tasks)
    span(ttask[k,t],[load[k,t],ship[k,t],unload[k,t]]);|\label{cp:span_constr}|
    startOf(ship[k,t]) == endOf(load[k,t])|\label{cp:ship_constr}|
    startOf(unload[k,t]) == endOf(ship[k,t])|\label{cp:unload_constr}|
  forall(k in Trucks)
    alternative(task[t],all(k in Trucks) ttask[k,t])|\label{cp:alt_constr}|
  forall(k in Trucks,t in Tasks)
    emptyMilesCost[k,t] = travelCost[destination[t],typeOfNext(truckSeq[k],ttask[k,t],destination[t],destination[t])];|\label{cp:empty_miles_single}|
  forall(k in Trucks)
    truckEmptyMilesCost[k] = sum(t in Tasks) emptyMilesCost[k,t];|\label{cp:empty_miles_total}|
  forall(k in Trucks)
    noOverlap(truckSeq,travelTime);|\label{cp:no_overlap_constr}|
}|\vspace{-0.25cm}|
\end{lstlisting}
\end{varwidth}
\end{lrbox}

\begin{figure}[!t]
\makebox[\textwidth][l]{%
\fbox{\begin{minipage}{0.47\textwidth}
	\usebox{\modelbox}
\end{minipage}}
}
\caption{CP Model for the ATHN Problem.}
\label{fig:formulation}
\end{figure}

The main decision variables are the interval variables {\tt task[t]}
that specify the start and end times of task {\tt t} when processed
by the autonomous network, and the optional interval variables {\tt
	ttask[k,t]} that are present if task {\tt t} is transported by
truck {\tt k}. These optional variables consist of three subtasks that
are captured by the interval variables {\tt load[k,t]} for loading,
{\tt ship[k,t]} for transportation, and {\tt unload[k,t]} for
unloading. The other key decision variables are the sequence variables
{\tt truckSeq[k]} for every truck: these variables
represent the sequence of tasks performed by every truck. They
contain the loading, shipping, and unloading interval variables
associated with the trucks, and their types. The type of a loading
interval variable is the origin of the task, the type of an unloading
interval variable is the destination of the task, and the type of the
shipping interval variable is the specific type {\tt shipType} that is
used to represent the fact that there is no transition cost and
transition time between the load and shipping subtasks, and the
shipping and destination subtasks. The model also contains two
auxiliary decision variables to capture the empty mile cost between a
task and its successor, and the empty mile cost of the truck
sequence.

The objective function (line \ref{cp:obj}) minimizes the total empty mile cost.
The constraints in lines \ref{cp:flex_constr_start}--\ref{cp:flex_constr_end} specify the potential start times of the tasks, and are defined in terms of the pickup times and the flexibility parameter.
The {\sc span} constraints (line \ref{cp:span_constr}) link the task variables and their subtasks, while the constraints in lines \ref{cp:ship_constr}--\ref{cp:unload_constr} link the subtasks together.
The {\sc alternative} constraints on line \ref{cp:alt_constr} specify that each task is processed by a single truck.
The empty mile costs between a task and its subsequent task (if it exists) is computed by the constraints in line \ref{cp:empty_miles_single} that use the {\sc 	typeOfNext} expression on the sequence variables.
The total empty mile cost for a truck is computed in line \ref{cp:empty_miles_total}.
The {\sc noOverlap} constraints in line 51 impose the disjunctive constraints between the tasks and the transition times.
The CP model is solved with the CPLEX CP Optimizer version 12.8.

\subsection{Vehicle Routing Modeling}
\label{sec:VRPFLcg}
The ATHN optimization can be modeled as a variant of the Vehicle Routing Problem with Full Truckloads \cite{ArunapuramEtAl2003-VehicleRoutingScheduling}.
The VRPFL is formulated on the directed \emph{task graph} $G=(V,A)$.
The set $V = T \cup \{src, snk\}$ contains a vertex for every task, together with a source and a sink node.
Operations of a single truck are modeled by a route from $src$ to $snk$, where arcs model the transition from one task to the next.
Arcs are defined between the tasks, going out of the source, and going into the sink.
If arc $a \in A$ connects task $t$ to task $t'$, it is associated with time $\tau_{tt'} = \tau_{o(t)d(t)} + 2S + \tau_{d(t)o(t')}$ and cost $c_{tt'} = c_{o(t)d(t)} + c_{d(t')o(t')}$, i.e., performing task $t$ and relocating to the start location of task $t'$.
Arcs connecting to $src$ take no time and have no cost, and arcs into $snk$ do not require relocation.
Every vertex $t \in V\backslash\{src,snk\}$ has a time window $[p(t)-\Delta, p(t)+\Delta]$.
The ATHN problem amounts to finding a minimum-cost set of at most $\lvert K \rvert$ feasible routes on $G$ that cover all vertices.
A route is feasible if it starts at the source, ends at the sink, and satisfies the time constraints.

\paragraph{Preprocessing} A computational challenge in solving Vehicle Routing Problems (VRPs) is dealing with cycles in the underlying graph: either cycles are not addressed, which results in a weak lower bound, or cycles are avoided, which takes computational effort \cite{CostaEtAl2019-ExactBranchPrice}. For long-distance trucking, it turns out that almost all cycles can be removed in preprocessing by only keeping arcs between task $t$ and $t'$ if it is possible to perform the tasks in that order, i.e., $p(t)-\Delta + \tau_{tt'} \le p(t')+\Delta$.
The arcs imply an ordering on $p(t)$ if all tasks take sufficiently long.
In particular, for $\Delta \le S$ it is guaranteed that $\tau_{tt'} > 2S \ge 2\Delta$ such that every arc follows the ordering $p(t) < p(t')$, and the graph is acyclic.

\paragraph{Column Generation}
VRP variants are often formulated with a set-partitioning formulation and solved with column generation (e.g., see \cite{CostaEtAl2019-ExactBranchPrice}).
Let $R$ be the set of all feasible routes, and let binary variable $x_r$ be one if and only if route $r \in R$ is selected.
The cost of a route $c_r$ is the sum over the arc costs.
Figure~\ref{fig:cgmodel} states the set-partitioning formulation.
Objective~\eqref{eq:cg:obj} minimizes the total distance and Constraints~\eqref{eq:cg:cover} ensure that all tasks are performed.
Equations~\eqref{eq:cg:x} define the variables.

\begin{figure}[!t]
	\centering
	\begin{align}
		\min \quad &\sum_{r \in R} c_r x_r\label{eq:cg:obj}\\
		%% Constraint %%
		\text{s.t.} \quad
		%% Constraint 1 %%
		&\sum_{r \in R \vert t \in r} x_{r} = 1 \qquad \forall t \in T \label{eq:cg:cover}\\
		&\sum_{r \in R} x_{r} \leq \lvert K \rvert\label{eq:cg:vehicles}\\
		&x_{r} \in \{0,1\} \qquad \forall r \in R \label{eq:cg:x}
	\end{align}
	\caption{Set-Partitioning Form. for the ATHN Problem.}
	\label{fig:cgmodel}
\end{figure}

A CG-based restricted master heuristic \cite{JoncourEtAl2010-ColumnGenerationBased} is used to solve the problem.
First, CG is used to solve the Linear Programming (LP) relaxation and to obtain a lower bound on the objective value.
CG is a technique to solve LPs by only generating the variables (columns) as they are needed, which makes it suited to deal with the large number of $x$-variables \cite{DesaulniersEtAl2005-ColumnGeneration}.
The problem is split into a master problem and a subproblem that are solved iteratively until convergence.
The master problem is an LP that is solved with Gurobi 9.1.2, and the subproblem is typically solved with a labeling algorithm.
This paper uses the labeling algorithm provided by the cspy Python package \cite{Sanchez2020-CspyPythonPackage}.
Motivated by the (almost) acyclic nature of the task graphs, this paper does not address cycles.

The CG procedure results in a valid lower bound and a set $\bar{R}$ of routes that are relevant to the LP relaxation.
An upper bound is generated by solving the set-partitioning formulation with integer variables for route set $\bar{R}$.
To make it easier to construct solutions, tasks are allowed to be performed multiple times, and duplicates are removed afterwards.
The result is a feasible solution to the ATHN problem.
The CG method is a heuristic, but may be extended to an exact method by embedding CG in a branch-and-bound framework \cite{BarnhartEtAl1998-BranchPriceColumn}.

\paragraph{Network Flow}
\label{sec:mcnf}
The VRPFL can also be modeled with the vehicle-flow formulation that is presented by Figure~\ref{fig:nfmodel} (similar to VRP1 in Toth and Vigo~\shortcite{TothVigo2014-VehicleRoutingProblems}).
Rather than using route variables, the vehicle-flow formulation uses binary flow variables $y_a$ that indicate whether arc $a \in A$ is part of any route.
For brevity, let $\delta^+(t)$ and $\delta^-(t)$ denote the set of arcs going out and coming into $t \in V$, respectively.
Objective~\eqref{eq:nf:obj} minimizes the total distance.
Constraints~\eqref{eq:nf:flow-cover} and~\eqref{eq:nf:flow+cover} require that each task has an inflow and outflow of one, which ensures the task is performed, and Constraint~\eqref{eq:nf:maxvehicles} limits the maximum number of vehicles.
The subtour elimination constraints prevent cyclic flows and the time constraints ensure that the time windows are satisfied.
Together the constraints ensure a disjoint set of feasible routes, just like the set-partitioning model in Figure~\ref{fig:cgmodel}.
Toth and Vigo~\shortcite{TothVigo2014-VehicleRoutingProblems} present different implementations for Constraints~\eqref{eq:nf:subtour}-\eqref{eq:nf:time} but the details are not important to this paper.

\begin{figure}[!t]
	\centering
	\begin{align}
		\min \quad &\sum_{a \in A} {c}_{a} y_{a} \label{eq:nf:obj}\\
		%% Constraint %%
		\text{s.t.} \quad
		%% Constraint 1 %%
		&\sum_{a \in \delta^{-} (t)}y_{a} = 1 \quad \forall t \in V\backslash\{src,snk\} \label{eq:nf:flow-cover}\\
		&\sum_{a \in \delta^{+} (t)}y_{a} = 1 \quad \forall t \in V\backslash\{src,snk\} \label{eq:nf:flow+cover}\\
		&\sum_{a \in \delta^{+} (src)}y_{a} \le \lvert K \rvert \label{eq:nf:maxvehicles}\\
		& \textrm{\emph{(subtour elimination constraints)}} \label{eq:nf:subtour}\\
		& \textrm{\emph{(time constraints)}} \label{eq:nf:time}\\
		&y_a \in \{0,1\} \quad \forall a \in A \label{eq:nf:vars}
	\end{align}
	\caption{Vehicle-Flow Form. for the ATHN Problem.}
	\label{fig:nfmodel}
\end{figure}

Given the near-acyclicity of task graphs in ATHNs,  Constraints~\eqref{eq:nf:subtour}-\eqref{eq:nf:time} are relaxed.
The remaining Problem~\eqref{eq:nf:obj}-\eqref{eq:nf:maxvehicles}, \eqref{eq:nf:vars} will be referred to as the \emph{NF model}.
Relaxing these constraints is motivated by the structure of the task graph: Only few cycles are expected {\em after preprocessing the arcs}, which makes the subtour elimination constraints almost redundant. Furthermore, the remaining arcs only connect tasks that can be performed sequentially in time, making the time constraints less important. Another important observation is that the NF model has the integrality property and can therefore be solved in polynomial time with LP \cite{AhujaEtAl1993-NetworkFlowsTheory}.

Solving the NF model \eqref{eq:nf:obj}-\eqref{eq:nf:maxvehicles}, \eqref{eq:nf:vars} with linear programming immediately provides a lower bound. This follows from the fact that the NF model is a relaxation of the vehicle-flow formulation~\eqref{eq:nf:obj}-\eqref{eq:nf:vars}.
It is important to remark that {\em the lower bound still depends on the time flexibility parameter $\Delta$}, even if the time constraints are relaxed. This stems from the arc preprocessing step, which removes more arcs when the flexibility shrinks.
The NF model for a flexibility $\delta$ is denoted by NF$_\delta$.

Upper bounds are generated according to the following procedure.
Solve NF$_\delta$ for a selection of different time flexibilities $\delta \le \Delta$, including $\delta=0$.
For each $\delta$ take the vehicle routes that are produced by NF$_\delta$ and try to apply them to NF$_\Delta$.
First, check if the subtour elimination constraints are satisfied.
Next, if this is the case, try to satisfy the time constraints of NF$_\Delta$ by following each of the routes and setting the earliest possible starting time for every vertex.
If this is also successful, the resulting solution is a feasible solution to the ATHN problem with flexibility $\Delta$.

The above procedure is opportunistic, but surprisingly it is \emph{guaranteed} to produce a feasible solution for $\delta=0$ if one exists, and this solution is trivially feasible for larger flexibilities.
\begin{proposition}
	For time flexibility $\Delta =  0$ the NF model produces an optimal solution.
\end{proposition}
\begin{proof}
	Arc preprocessing for $\Delta=0$ ensures that all arcs $(t,t')$ satisfy $p(t) + \tau_{tt'} \le p(t')$.
	This immediately implies that the graph is acyclic and thus that the subtour elimination constraints are automatically satisfied.
	Following any route, the condition above ensures that starting task $t$ at time $p(t)$ satisfies the time constraints.
	As the minimum-cost solution to NF$_0$ also satisfies the relaxed constraints, the solution is optimal.
\end{proof}

\noindent
The fact that the upper bound procedure is guaranteed to work is specifically due to autonomous vehicles.
The main technical difference is that autonomous trucks are completely interchangable, while human-driven trucks need to be distinguished to ensure that drivers return to their specific starting point \cite{ArunapuramEtAl2003-VehicleRoutingScheduling} or that they do not exceed the maximum driving time \cite{GronaltEtAl2003-NewSavingsBased}.
Gronalt \emph{et al.}~\shortcite{GronaltEtAl2003-NewSavingsBased} consider a network flow relaxation with aggregated drivers which results in a lower bound, but the outcomes cannot be transformed into upper bounds.
Human factors do not apply to autonomous trucks, which enables the simple upper bound strategy in this paper.

\section{Case Study}
\label{sec:casestudy}
The three models are applied to the realistic order data introduced in the Ryder~\shortcite{RyderSAM2021-ImpactAutonomousTrucking} white paper.
Ryder prepared this representative dataset for its dedicated transportation business in the Southeast of the US, reducing the scope to orders that were strong candidates for automation.

\paragraph{Data Description}
\label{sec:data}
The dataset consists of long-haul trips (431 miles average) that start in the first week of October 2019, and stay completely within the following states: AL, FL, GA, MS, NC, SC, and TN.
The case study focuses on scheduling the 494 most \emph{challenging orders} that currently consist of a single delivery followed by an empty return trip.
These orders make up 24\% of the dataset, and account for 53\% of the empty mileage.
Two sets of hubs are considered: a small network with 17 transfer hubs in the areas where Ryder trucks most frequently access the highway system, and a large network that includes 13 additional hubs in locations with fewer highway access points.
Figure~\ref{fig:designs} visualizes the networks.
The exact hub locations are masked, but the figures are accurate within a 50 mile range.
Orders that are better served with a conventional truck are filtered out, and the remaining orders are split into separate tasks.
See the Ryder \shortcite{RyderSAM2021-ImpactAutonomousTrucking} white paper for additional details.

\begin{figure}[!t]
	\centering
	\includegraphics[width=0.20\textwidth, trim=28cm 3cm 29cm 15cm, clip]{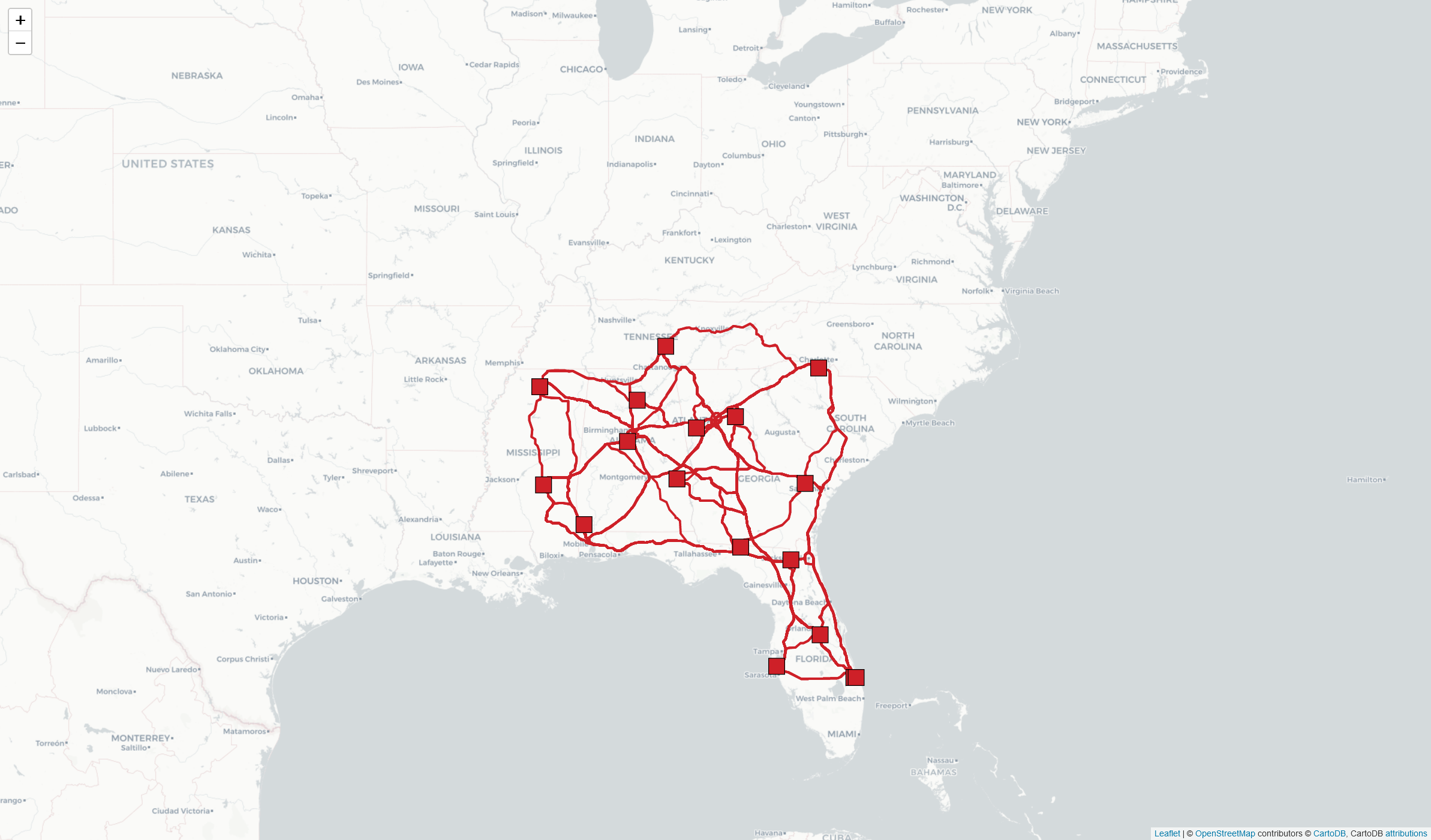}
	\hspace{0.5cm}
	\includegraphics[width=0.20\textwidth, trim=28cm 3cm 29cm 15cm, clip]{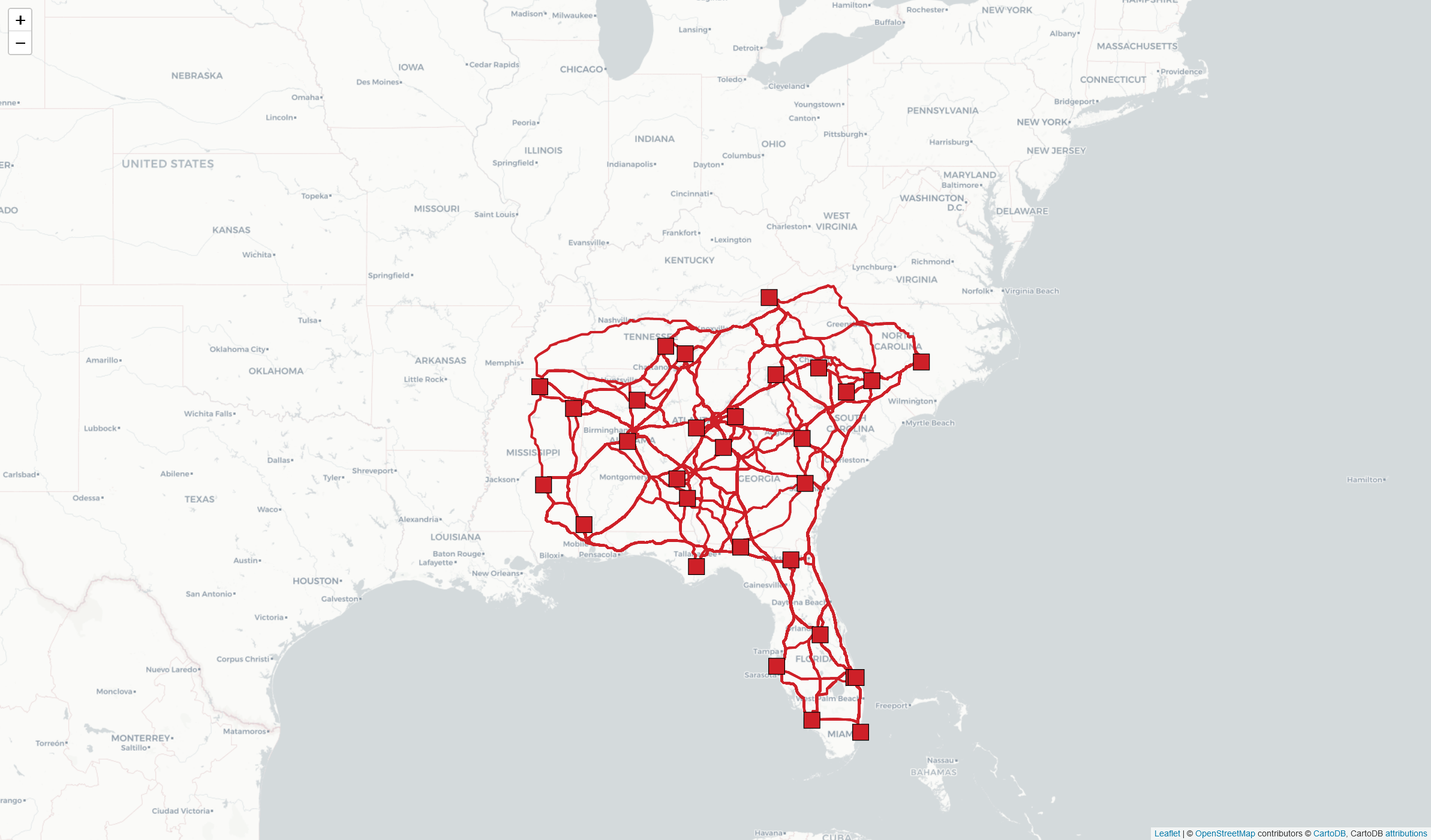}
	\caption{Small and Large Networks for the Southeast.}
	\label{fig:designs}
\end{figure}

\paragraph{Base Cases}
\label{sec:base}
Two base cases are considered in the case study: The \emph{N-17 base case} on the small network and the \emph{N-30 base case} on the large network.
Both cases assume a time flexibility of $\Delta=60$ minutes, loading or unloading of $S=30$ minutes, $\lvert K \rvert = 50$ autonomous trucks are available, and operating autonomous trucks will be 25\% cheaper per mile than conventional trucks.
After filtering, the N-17 and N-30 base cases respectively serve 437 and 468 orders on the ATHN.
This paper focuses on optimizing the autonomous part of the system.
When comparing costs to the current system, it is assumed that the first/last-mile tasks can be served with at most 25\% empty miles.
The associated models are significant in size: The CP model (Figure~\ref{fig:formulation}) has about 100,000 variables and 100,000 constraints.
The NF model~\eqref{eq:nf:obj}-\eqref{eq:nf:maxvehicles}, \eqref{eq:nf:vars} is similar in size but only contains continuous variables.
The set-partitioning formulation (Figure~\ref{fig:cgmodel}) has around 500 constraints and generates about 3,000 variables as part of the CG procedure.

\paragraph{Base Case Results}
\label{sec:baseresults}
The CG and NF methods are used to generate the lower bounds that are presented by Table~\ref{tab:basecaseLB}.
The NF model is indeed very effective in exploiting the problem structure: Compared to CG, the solution time is reduced from hours to seconds while the lower bounds remain exactly the same.
This result is possible because the NF model only ignores constraints that are unlikely to be violated, and therefore hardly affect the bound.
Calculating the NF bound only requires solving an LP, which explains the tremendous speedup.
The bounds are not guaranteed to always be the same: If autonomous trucks are 40\% cheaper than regular trucks and time flexibility is increased to $\Delta=120$ minutes, then CG generates a lower bound that is 0.07\% better than the bound produced by NF.
Hence the NF model presents an excellent trade-off in terms of quality and time.
The Ryder~\shortcite{RyderSAM2021-ImpactAutonomousTrucking} white paper only considers the CP method, which inherently does not produce bounds.
Now that lower bounds are available they will be used to assess solution quality and complement earlier work.

\begin{table}[!t]
    \begin{adjustbox}{width=1\linewidth,center}
        \begin{tabular}{lcccccc}
        \toprule
        & \multicolumn{2}{c}{CP} & \multicolumn{2}{c}{CG} & \multicolumn{2}{c}{NF} \\
        \cmidrule(lr){2-3} \cmidrule(lr){4-5} \cmidrule(lr){6-7}
         & LB (mi)    & Time & \multicolumn{1}{c}{LB (mi)} & \multicolumn{1}{c}{Time} & \multicolumn{1}{c}{LB (mi)} & \multicolumn{1}{c}{Time} \\
        \midrule
        N-17 & n.a.   & n.a.   & 122,061  & 17.7 h  & 122,061 & 8 s  \\
        N-30 & n.a.   & n.a.   & 134,382  & 24.1 h  & 134,382 & 9 s  \\
        \bottomrule
        \end{tabular}
    \end{adjustbox}
    \caption{Base Case Lower-Bound Comparison.}
    \label{tab:basecaseLB}
    \vspace{\baselineskip}
    \begin{adjustbox}{width=1\linewidth,center}
        \begin{tabular}{lcccccc}
            \toprule
            & \multicolumn{2}{c}{CP} & \multicolumn{2}{c}{CG} & \multicolumn{2}{c}{NF} \\
            \cmidrule(lr){2-3} \cmidrule(lr){4-5} \cmidrule(lr){6-7}
             & UB (mi)    & Time & \multicolumn{1}{c}{UB (mi)} & \multicolumn{1}{c}{Time} & \multicolumn{1}{c}{UB (mi)} & \multicolumn{1}{c}{Time} \\
            \midrule
            N-17 & 135,834 (11.3\%) & 1.0 h  & -  & 1.0 h  & 122,658 (0.5\%) & 20 s  \\
            N-30 & 150,573 (12.0\%) & 1.0 h  & -  & 1.0 h  & 135,486 (0.8\%) & 20 s  \\
            \bottomrule
        \end{tabular}
    \end{adjustbox}
    \caption{Upper-Bound Comparison (gaps in parenthesis).}
    \label{tab:basecaseUB}
\end{table}

Table~\ref{tab:basecaseUB} compares the three methods on their ability to produce high-quality solutions.
For consistency with earlier work, each method is given one additional hour after computing the lower bound to generate a feasible solution.
The CG method makes use of the route set $\bar{R}$ that was obtained while calculating the lower bound.
The NF method uses the procedure described in Section~\ref{sec:mcnf} to opportunistically generate solutions from NF$_\delta$ for $\delta\in \{0, 30, 60\}$.
Additionally, the table includes optimality gaps that compare the upper bounds to the lower bounds presented earlier.

\begin{figure}[!t]
	\includegraphics[width=0.9\linewidth,trim={0.5cm 1.2cm 5cm 0},clip]{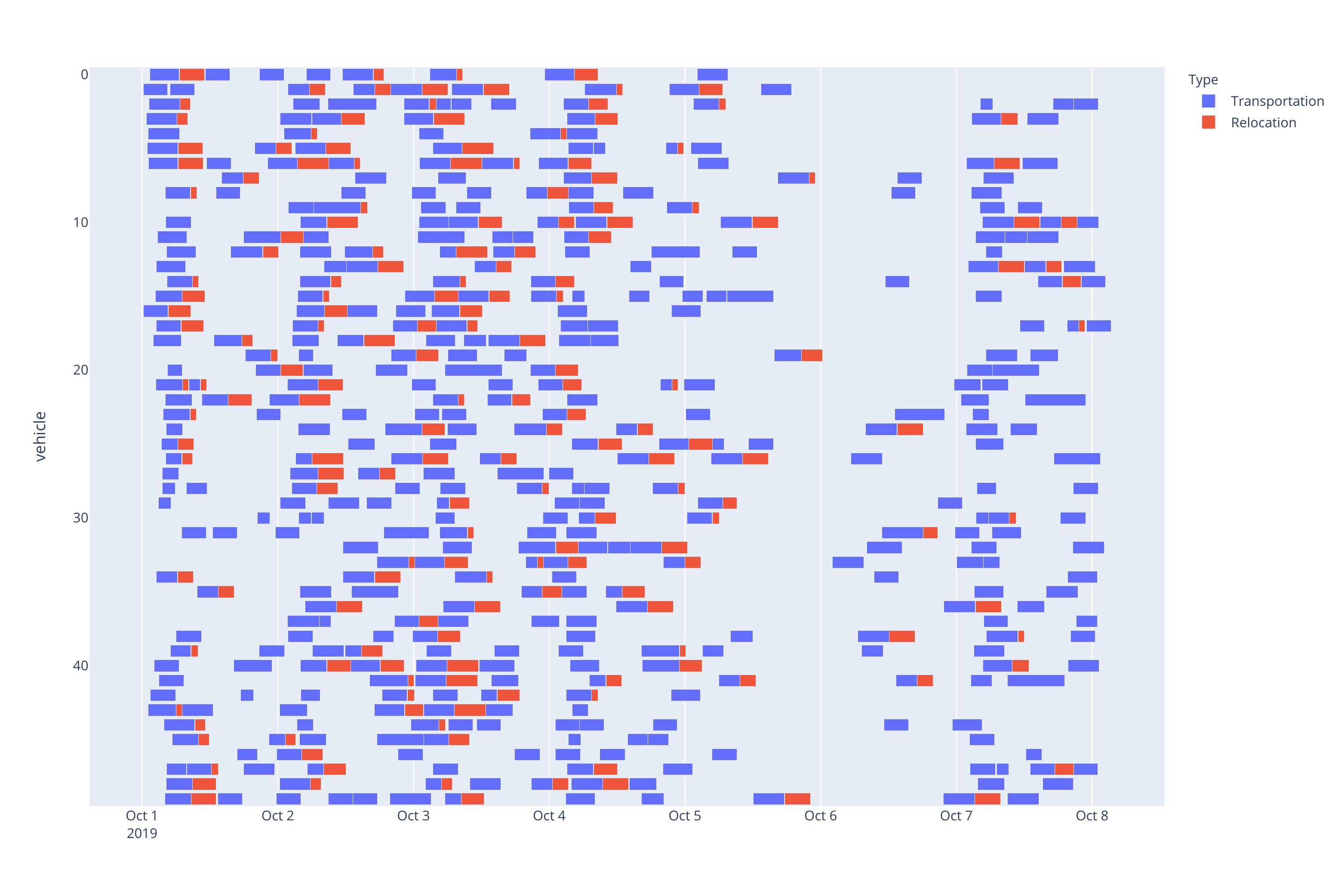}
	\caption{NF Solution for the N-17 Base Case (one vehicle per row, blue for performing tasks and red for relocation).}
	\label{fig:small_ub_gantt}
\end{figure}

The NF method again outperforms the other methods.
The solutions found by the CP method already correspond to significant savings of more than 27\% for ATHN compared to the current system.
However, the optimality gaps of over 11\% show that more savings may be possible.
Surprisingly, the CG method fails to find a feasible solution within one hour.
Intermediate solutions are found with 55 and 56 vehicles for N-17 and N-30 respectively, but these plans are already more expensive than those found with CP.
It seems that the route set $\bar{R}$ is not necessarily a good starting point to find a feasible solution.
The NF method is the clear winner: Not only are the solutions significantly better than those obtained from the CP method, they are provably within only 1\% from optimality.
Figure~\ref{fig:small_ub_gantt} visualizes the NF solution for the N-17 base case: The tasks are spread out over the week, which explains why the time constraints are not very restrictive and relaxing them still leads to good results.

\paragraph{Impact of Time Flexibility}
\label{sec:delta}
\begin{figure}[!t]
	\centering
	\includegraphics[scale=0.42]{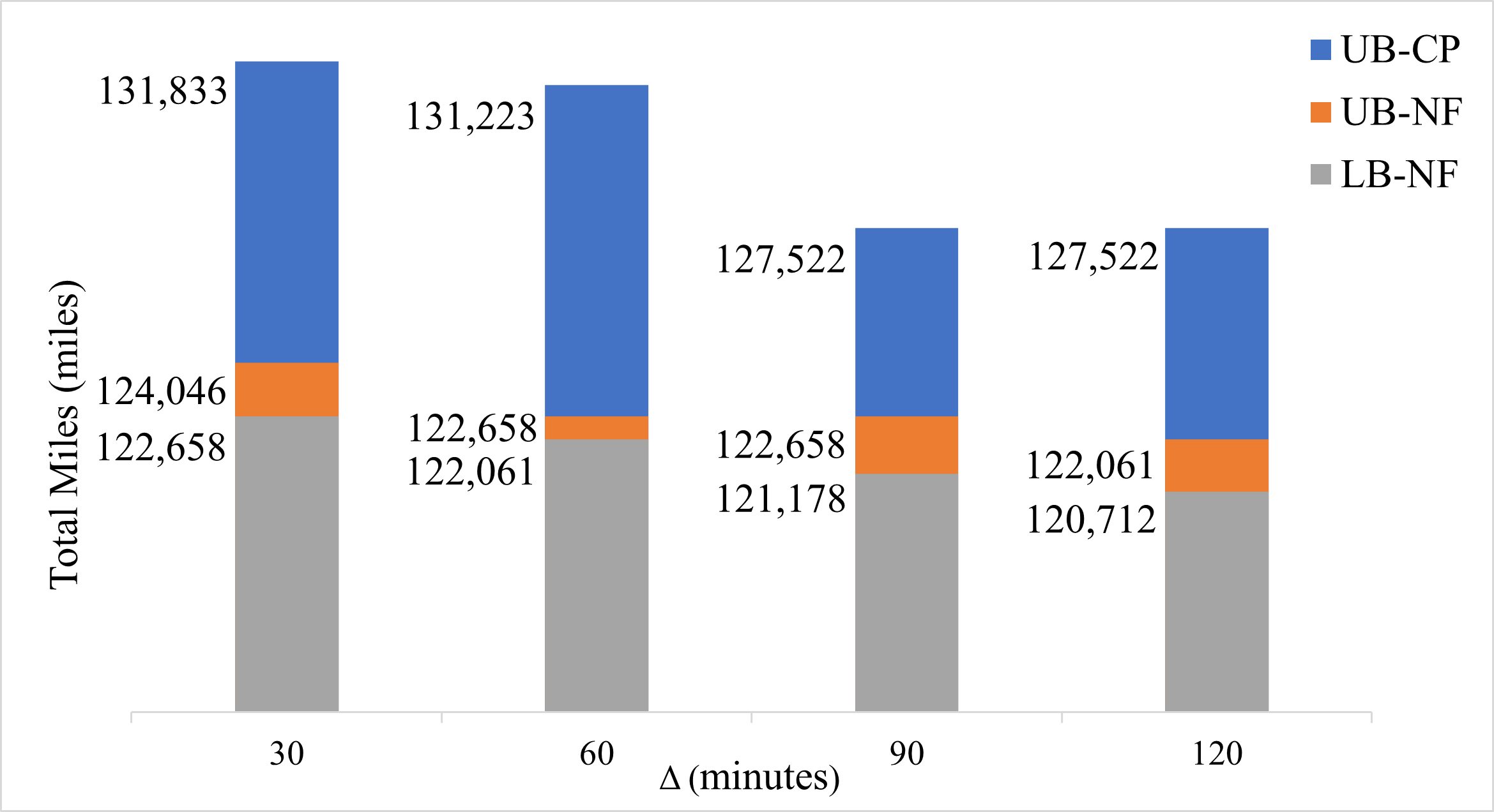}
	\caption{Impact of Time Flexibility.}
	\label{fig:twdelta}
\end{figure}

Deviating from an agreed time window flexibility must be negotiated with the customer, but if there are significant benefits in terms of efficiency, this may be worth the effort.
To determine the impact of appointment flexibility, the two models that produced feasible solutions (CP and NF) are compared on the N-17 base case for different values of $\Delta$ ranging from 30 to 120 minutes.
The time limit is set to four hours per setting to account for the fact that instances with more flexibility are more difficult to solve.
Figure~\ref{fig:twdelta} summarizes the upper bounds (UB-CP, UB-NF) and the lower bounds (LB-NF) that are generated for different flexibility parameters.
If UB-CP deteriorates when $\Delta$ increases, the previous better solution is reported.
The upper bounds for the NF method are calculated with the opportunistic procedure for $\delta \in \{0,30,60,90,120\}$.

Similar to the base case, Figure~\ref{fig:twdelta} shows that the NF method maintains strong performance when time flexibility is increased.
The largest optimality gap of 1.2\% is obtained for $\Delta=90$.
The NF method was unable to improve from $\Delta=60$ to $\Delta=90$, but is able to benefit from additional flexibility provided by $\Delta=120$.
The CP method is able to benefit significantly from the extension of $\Delta=60$ to $\Delta=90$, but does not improve after that.
A possible explanation is that additional flexibility increases the search space, which makes it more difficult for the CP method to find a good solution.

The performance of the NF method is consistent when the experiment is repeated for different values of cost reduction per mile compared to conventional trucks.
Additional instances are generated for 30\%, 35\%, and 40\% cost reduction and for different time flexibilities.
On all instances the optimality gap found by the NF method is less than 1.6\%.
These results further support that the NF method is able to find high-quality solutions.

\paragraph{Impact on ATHN Cost Savings}
The Ryder \shortcite{RyderSAM2021-ImpactAutonomousTrucking} white paper reported cost saving in the range of 27\%-40\% for the challenging orders when ATHN is compared to current operations.
The lower value stems from applying the CP model to the N-17 base case.
The solution found by the NF model improves this number to 32.2\%.
The higher value follows from the scenario that considers the N-30 base case with autonomous trucks that are 40\% cheaper per mile.
Applying the NF model to this instance results in a solution with 44.0\% cost savings compared to current operations.
It follows that better optimization methods improve the range of potential cost savings from 27\%-40\% to 32\%-44\%, which strengthens the business case for ATHNs and autonomous trucking.

\section{Conclusion}
\label{sec:conclusion}
Autonomous trucks are expected to fundamentally transform the freight transportation industry.
Recent studies have shown that Autonomous Transfer Hub Networks (ATHN), which combine autonomous trucks on middle miles with human-driven on the first and last miles, can produce significant savings.
This paper presented three different methods to optimize ATHN operations: a Constraint-Programming (CP) method, a Column-Generation (CG) method, and a Network Flow (NF) method.
The methods were compared on a realistic case study with data provided by Ryder System, Inc.
The paper complemented earlier work by calculating lower bounds on performance.
This showed that the ATHN CP solution, which already corresponds to large savings compared to the current system, could potentially be improved.
It was demonstrated on the case study that the NF model effectively exploits the problem structure and outperforms the other methods: It produces both lower bounds that are comparable to the CG method and upper bounds that improve upon the CP method in a matter of seconds.
Further analysis showed that NF is able to benefit from additional pickup flexible, and consistently outperforms the other methods.
The NF method improved the range of potential savings of ATHN from 27\%-40\% to 32\%-44\%, further strengthening the business case for autonomous trucking.

\section*{Acknowledgments}
This research was partly funded through a gift from Ryder and partly supported by the NSF AI Institute for Advances in Optimization (Award 2112533). Special thanks to the Ryder team for their invaluable support, expertise, and insights.

%% The file named.bst is a bibliography style file for BibTeX 0.99c
\bibliographystyle{named}
\bibliography{references}

\end{document}